\documentclass[preprint,12pt]{elsarticle}

\usepackage{amssymb,amsmath,lineno,hyperref,geometry}
\numberwithin{equation}{section}
\newtheorem{theorem}{Theorem}[section]
\newtheorem{lemma}{Lemma}[section]

\newtheorem{definition}{Definition}[section]

\newtheorem{proposition}{Proposition}[section]

\journal{Journal}

\begin{document}

\newcommand\tbbint{{-\mkern -16mu\int}}
\newcommand\tbint{{\mathchar '26\mkern -14mu\int}}
\newcommand\dbbint{{-\mkern -19mu\int}}
\newcommand\dbint{{\mathchar '26\mkern -18mu\int}}
\newcommand\bint{
{\mathchoice{\dbint}{\tbint}{\tbint}{\tbint}}
}
\newcommand\bbint{
{\mathchoice{\dbbint}{\tbbint}{\tbbint}{\tbbint}}
}

\begin{frontmatter}



\title{{\bf Global existence and finite time blowup for a mixed pseudo-parabolic $p$-Laplacian type equation}}

\author{Jiazhuo Cheng}
\ead{chengjzh5@mail2.sysu.edu.cn}
\author{Qiru Wang\corref{cor1}
}
\ead{mcswqr@mail.sysu.edu.cn}

\cortext[cor1]{Corresponding author}
\address{School of Mathematics, Sun Yat-sen University, Guangzhou 510275, Guangdong, China}


\begin{abstract}
This paper concerns the initial-boundary value problem for a mixed pseudo-parabolic $p$-Laplacian type equation. By constructing a family of potential wells, we first present the explicit expression for the depth of potential well, and then prove the existence, uniqueness and decay estimate of global solutions and the blowup phenomena of solutions with subcritical initial energy. Next, we extend parallelly these results to the critical initial energy. Lastly, the existence, uniqueness and asymptotic behavior of global solutions with supercritical initial energy are proved by further analyzing the properties of $\omega$-limits of solutions.

\end{abstract}


\end{frontmatter}


\section{Introduction}
\label{section1}

In this paper, we consider the following semilinear pseudo-parabolic equation with initial-boundary conditions
\begin{eqnarray}\label{1.1}
\left\{\begin{split}
& u_t-\mathrm{div}\left(\left|\nabla u\right|^{p-2}\nabla u\right)-\triangle u_t=\left|u\right|^{q-1}u-\bbint_\Omega\left|u\right|^{q-1}u\operatorname dx,~~ &\mathrm{in}\;\Omega\times\left(0,T\right),\\
&\frac{\partial u}{\partial\nu}\left(x,t\right)=0,   &\mathrm{on}\;\partial\Omega\times\left(0,T\right),\\
& u\left(x,0\right)=u_0\left(x\right)\not\equiv0,~~\bbint_\Omega u_0\operatorname dx=0,  &\mathrm{in}\;\Omega,
\end{split}\right.
\end{eqnarray}
where, (1) the first equation of (\ref{1.1}) is a mixed pseudo-parabolic $p$-Laplacian type equation, $p>1, \;\max\left\{p-1,1\right\}<q<\frac{np}{n-p}-1$, $T\in(0,\infty\rbrack$, $\Omega\subset \mathbb{R}^n\left(n\geq1\right)$ is a bounded domain with smooth boundary, $\bbint_\Omega u\operatorname dx=\frac1{\left|\Omega\right|}\int_\Omega u\operatorname dx$ and $\bbint_\Omega\left|u\right|^{q-1}u\operatorname dx$ is to conserve the spatial integral of
the unknown function over time; (2) the second equation of (\ref{1.1}) is Neumann boundary condition, $\frac\partial{\partial\nu}$ denotes differentiation with respect to the outward normal $\nu$ on $\partial \Omega$; (3) the third equation of (\ref{1.1}) is initial boundary condition.

Problem (\ref{1.1}) has attracted much attentions in recent years in the fields of population dynamics and biological sciences \cite{1,2,3,4}, where the total mass is often conserved or known. Such equations can provide insight into biological and chemical problems where conservation properties predominate. Since in mathematics we do not require $u(x,y)$ to be non-negative, we use $\left|u\right|^{q-1}u$ instead of $u^q$ in the problem (\ref{1.1}) (see \cite{5}). In addition, when affected by a variety of factors, such as atoms, ions, etc., the $\triangle u$ term in the pseudo-parabolic is usually replaced by $\mathrm{div}\left(\left|\nabla u\right|^{p-2}\nabla u\right)$.

The pseudo-parabolic equation
\begin{equation}\label{1.4}
u_t-u_{xxt}=F\left(x,t,u_x,u_{xx}\right)
\end{equation}
has been extensively studied by many authors \cite{6,7,8,9,10}. From 2013 to 2018, the authors in \cite{11,12,13,14} investigated the initial boundary value problem  (\ref{1.4}) with $F\left(x,t,u_x,u_{xx}\right)=u_{xx}+u^p$. They proved global existence and asymptotic behavior of solutions with subcritical and critical initial energy $J(u_0)\leq d$, further gave the global nonexistence of solutions under supercritical initial energy $J(u_0)>d$ through comparison principle, and the upper bound of the blowup time of supercritical initial energy is estimated. In 2021, as an expansion of previous studies, Wang and Xu \cite{15} considered the following semilinear pseudo-parabolic equation with Neumann boundary condition
\begin{equation*}
u_t-\triangle u-\triangle u_t=\left|u\right|^{p-1}u-\bbint_\Omega\left|u\right|^{p-1}u\operatorname dx,
\end{equation*}
and proved global existence, uniqueness, asymptotic behavior and blowup of solutions with subcritical and critical initial energy $J(u_0)\leq d$.

In 2007, Yin and Jin  \cite{16} studied the following equation
\begin{equation}\label{1.2}
u_t-\mathrm{div}\left(\left|\nabla u^m\right|^{p-2}\nabla u^m\right)=\lambda\left|u\right|^q,
\end{equation}
and proved the finite time blowup when $m=1,\;1<p<2,\;q>1,\;\lambda>0$. In 2009, Jin et al. \cite{17} considered problem (\ref{1.2}) with $0<m\left(p-1\right)<1,\;q>0,\;\lambda>0$, and the finite time blowup was proved for $q>m\left(p-1\right)$. In 2014, Qu et al. \cite{18} studied the nonlocal $p$-Laplace equation
\begin{equation*}
u_t-\mathrm{div}\left(\left|\nabla u\right|^{p-2}\nabla u\right)=\left|u\right|^q-\frac1{\left|\Omega\right|}\int_\Omega^{}\left|u\right|^q\operatorname dx
\end{equation*}
with Neumann boundary condition, and proved that the solution blows up for special initial energy, i.e., $J(u_0)\leq0$.

Compared with the above studies, there are relatively few studies on the mixed pseudo parabolic $p$-Laplace equations, see \cite{19,20,21,22,23,Ding}. In 2003, Liu \cite{22} studied the following nonlinear pseudo parabolic $p$-Laplace equation
\begin{equation}\label{1.6}
u_t-k\triangle u_t-\mathrm{div}\left(\left|\nabla u\right|^{p-2}\nabla u\right)=f\left(u\right),
\end{equation}
and proved the existence and uniqueness of solutions for $p>2,\;k>0,\;f\left(u\right)=0$. In 2013, Li et al. \cite{23} extended the problem (\ref{1.6}) to $p\geq2,\;0<q<p-1,\;k>0$, and proved the existence and asymptotic behavior of solutions with $f\left(u\right)=u^q$. In 2018, Cao and Liu \cite{24} considered the equation (\ref{1.6}) for $1<p<2,\;k\geq0,\;f\left(u\right)=\left|u\right|^{p-2}u\log\left|u\right|$ and obtained the global existence and asymptotic behavior of  solutions by constructing a family of potential wells.

In order to overcome the difficulties encountered in the prior estimation of solutions using Galerkin method, we use the potential well method proposed by Sattinger \cite{25} in 1968. Liu et al. \cite{26,27} extended and improved the method by introducing a family of potential wells and taking the known potential wells as a special case. Now, it is one of the most useful methods to prove the global existence and nonexistence of solutions and the vacuum isolation of solutions for parabolic equations \cite{28,29}.

It is interesting to consider (\ref{1.1}) due to the role of the corresponding conservation properties in the real world and its connection to the biological and chemical model equations (\ref{1.6}). In addition, keeping balance is an arduous task between the bad side and the good side of the strong dissipation term $\triangle u_t$, which can not only help the global existence by decaying energy, but also impede blowup and delay the blowup time. In this paper, it is important to deal with the blowup situation and maintain this balance, but it is not easy. Therefore, we combine the theory of potential wells with the Galerkin method to overcome the existing difficulties, and obtain the global existence, uniqueness, asymptotic behavior and finite time blow up of solutions.

The structure of this paper is organized as follows: In Section 2, we give some notations, definitions and lemmas about the basic properties of the related functionals and sets. Also, we present the main results of this paper. In Sections 3 and 4, we are devoted to the subcritical initial energy $J(u_0)<d$ and the critical initial energy $J(u_0)=d$, respectively, and prove the existence, uniqueness, decay estimate of global solutions and the blowup phenomena of solutions. In Section 5, we prove the global existence, uniqueness and asymptotic behavior of solutions for the supercritical initial energy $J\left(u_0\right)>d$ by analyzing the properties of $\omega$-limits of solutions. The conclusion is made in Section 6.

\section{Preliminaries and main results}

Throughout the whole paper, let $C$ represent generic positive constant, which may change from line to line. Let $L^p\left(\Omega\right)$ be the set of all measurable functions on $\Omega$ and satisfying $\int_\Omega\left|u\left(x\right)\right|^p\operatorname dx<\infty$. For $1\leq p\leq\infty$, denote by ${\left\|\cdot\right\|}_p$ the $L^p\left(\Omega\right)$ norm, i.e.,
\begin{equation*}
{\left\|u\right\|}_p=\left\{\begin{array}{l}\left(\int_\Omega\left|u\left(x\right)\right|^p\operatorname dx\right)^\frac1p,\;\;\;\mathrm{if}\;1\leq p<\infty;\\\mathrm{ess}\;\underset{x\in\Omega}{\mathrm{sup}}\left|u\left(x\right)\right|,\;\;\;\;\;\;\;\mathrm{if}\;p=\infty,\end{array}\right.
\end{equation*}
 where $\forall u\in L^p\left(\Omega\right)$. For $1\leq p\leq+\infty$, denote
 \begin{equation*}
W^{m,p}\left(\Omega\right)=\left\{u\in L^p\left(\Omega\right)|D^\alpha u\in L^p\left(\Omega\right),\forall\alpha\in Z_+^n,\left|\alpha\right|\leq m\right\},
 \end{equation*}
 and
  \begin{equation*}
W_N^{{m,p}}\left(\Omega\right)=\left\{u\in W^{m,p}\left(\Omega\right)|{\left.\frac{\partial u}{\partial\nu}\right|}_{\partial\Omega}=0,\int_\Omega u\operatorname dx=0\right\},
 \end{equation*}
where $D^\alpha u$ represents the $\alpha$ order weak derivative of $u$.

When $p=2$, we can write $H^k\left(\Omega\right)=W^{k,2}\left(\Omega\right)\left(k=0,1,\cdots\right)$. Let $\left(\cdot,\cdot\right)$ be the inner product in $L^2\left(\Omega\right)$ and ${\left\|\cdot\right\|}_{H^1}$ be the norm of $H^1\left(\Omega\right)$. That is
\begin{equation*}
\left(u,v\right)=\int_\Omega u\left(x\right)v\left(x\right)\operatorname dx,\;\;\;\forall u,v\in L^2\left(\Omega\right),
\end{equation*}
and
\begin{equation*}
{\left\|u\right\|}_{H^1}=\left(\left\|u\right\|_2^2+\left\|\nabla u\right\|_2^2\right)^\frac12,\;\;\;\forall u\in H^1\left(\Omega\right),
\end{equation*}
then the norm ${\left\|\cdot\right\|}_{H^1}$ is equivalent to the norm ${\left\|\nabla\left(\cdot\right)\right\|}_2$.

Next, we define the energy functional $J\left(u\right)$ and the Nehari functional $I\left(u\right)$ as
\begin{equation}\label{2.1}
\begin{array}{l}J\left(u\right)=\frac1p\left\|\nabla u\right\|_p^p-\frac1{q+1}\left\|u\right\|_{q+1}^{q+1},\\I\left(u\right)=\left\|\nabla u\right\|_p^p-\left\|u\right\|_{q+1}^{q+1},\end{array}
\end{equation}
and define the Nehari manifold as
\begin{equation*}
\mathcal N=\left\{\left.u\in W_N^{1,p}\left(\Omega\right)\right|I\left(u\right)=0,{\left\|\nabla u\right\|}_p\neq0\right\}.
\end{equation*}

By $J\left(u\right)$ and $I\left(u\right)$, we set
\begin{equation*}
\begin{array}{l}W=\left\{\left.u\in W_N^{1,p}\left(\Omega\right)\right|J\left(u\right)<d,I\left(u\right)>0\right\}\cup\left\{0\right\},\\V=\left\{\left.u\in W_N^{1,p}\left(\Omega\right)\right|J\left(u\right)<d,I\left(u\right)<0\right\},\end{array}
\end{equation*}
where
\begin{equation*}
d=\underset{u\in \mathcal N}{\mathrm{inf}}J\left(u\right)
\end{equation*}
is the depth of the potential well.

For $\delta>0$, we introduce a family of potential wells as follows
\begin{equation}\label{2.211}
\begin{array}{l}I_\delta\left(u\right)=\delta\left\|\nabla u\right\|_p^p-\left\|u\right\|_{q+1}^{q+1},\\{\mathcal N}_\delta=\left\{\left.u\in W_N^{1,p}\left(\Omega\right)\right|I_\delta\left(u\right)=0,{\left\|\nabla u\right\|}_p\neq0\right\},\\W_\delta=\left\{\left.u\in W_N^{1,p}\left(\Omega\right)\right|J\left(u\right)<d\left(\delta\right),I_\delta\left(u\right)>0\right\}\cup\left\{0\right\},\\V_\delta=\left\{\left.u\in W_N^{1,p}\left(\Omega\right)\right|J\left(u\right)<d\left(\delta\right),I_\delta\left(u\right)<0\right\},\\d\left(\delta\right)=\underset{u\in{\mathcal N}_\delta}{\mathrm{inf}}J\left(u\right).\end{array}
\end{equation}

In what follows, we give some sets and functionals for weak solutions with high energy levels. Set
\begin{equation*}
\begin{array}{l}
  {\mathcal N}_+=\left\{u\in W_N^{1,p}\left(\Omega\right)\left|I\left(u\right)>0\right.\right\},\\
  J^\alpha=\left\{u\in W_N^{1,p}\left(\Omega\right)\left|J\left(u\right)\leq\alpha\right.\right\},\\\mathcal N^\alpha=\mathcal N\cap J^\alpha=\left\{u\in\mathcal N\left|J\left(u\right)\leq\alpha\right.\right\},\;\;\;\forall\alpha>d,\\\lambda_\alpha=\inf\left\{\left\|u\right\|_{H^1}^2\left|u\in\right.\mathcal N^\alpha\right\},\;\;\;\forall\alpha>d.
  \end{array}
\end{equation*}
Obviously, $\lambda_\alpha$  is non-increasing with respect to $\alpha$.

Integrating the first equation of (\ref{1.1}) with respect to $x$ over $\Omega$ and then with respect to $t$ from $0$ to $t$, we obtain $\int_\Omega u\operatorname dx=\int_\Omega u_0\operatorname dx=0$. That is, we have the conservation law for problem (\ref{1.1}). Before stating the main theorems, we first give some definitions and lemmas.

As in \cite[Page 8]{Evans}, ``In general, as we shall see, the conservation law has no classical solutions, but {\it is} well-posed if we allow for properly defined {\it generalized} or {\it weak solutions}. This is all to say that we may be forced by the structure of the particular equation to abandon the search for smooth, classical solutions. We must instead, while still hoping to achieve the well-posedness conditions (a)-(c), investigate a wider class of candidates for solutions. And in fact, even for those PDE which turn out to be classically solvable, it is often most expedient initially to search for some appropriate kind of weak solution."

\begin{definition}[Weak solution \cite{Evans}]\label{2.1}
Function $u\left(x,t\right)$ is called a weak solution to problem (\ref{1.1}) on $\Omega\times\lbrack0,T)$, if $u\in L^\infty\left(0,T;W_N^{1,p}\left(\Omega\right)\right)$, $u_t\in L^2\left(0,T;W_N^{1,2}\left(\Omega\right)\right)$, $u\left(x,0\right)=u_0\left(x\right)\in W_N^{1,p}\left(\Omega\right)$ and satisfies
\begin{equation}\label{2.331}
\begin{array}{l}\int_0^t\left(\left(u_\tau,\varphi\right)+\left(\nabla u_\tau,\nabla\varphi\right)+\left(\left|\nabla u\right|^{p-2}\nabla u,\nabla\varphi\right)\right)\operatorname d\tau\\\;\;\;\;\;\;\;\;\;\;\;\;\;\;\;\;\;\;\;\;\;\;\;\;\;\;\;\;\;\;\;\;\;\;=\int_0^t\left(\left|u\right|^{q-1}u-\bbint_\Omega\left|u\right|^{q-1}u\operatorname dx,\varphi\right)\operatorname d\tau,\end{array}
\end{equation}
for any $\varphi\in W_N^{1,p}\left(\Omega\right)\cap W_N^{1,2}\left(\Omega\right)$.

Moreover, the following equality
\begin{equation*}
  \int_0^t\left\|u_\tau\right\|_{H^1}^2\operatorname d\tau+J\left(u\right)=J\left(u_0\right)
\end{equation*}
holds for $t\in\lbrack0,T)$.
\end{definition}

\begin{definition}[Maximal existence time]\label{2.2}
If $u\left(t\right)$ is a weak solution of problem (\ref{1.1}), for maximal existence time $T$ of $u\left(t\right)$, we have the following definition.

\emph{(i)} For $0\leq t<\infty$, if $u\left(t\right)$ exists, then $T=+\infty$;

\emph{(ii)} For $0<t_{0}<\infty$, if $0\leq t<t_0$, then $u\left(t\right)$ exists,  but doesn't exist at $t=t_0$, then $T=t_0$.
\end{definition}

\begin{lemma}[Relations between $I_\delta\left(u\right)$ and ${\left\|\nabla u\right\|}_p$]\label{2.2}
Let $u\in W_N^{1,p}\left(\Omega\right)$.

\emph{(i)} If $0<{\left\|\nabla u\right\|}_p<r\left(\delta\right)$, then $I_\delta\left(u\right)>0$;

\emph{(ii)} If $I_\delta\left(u\right)<0$, then ${\left\|\nabla u\right\|}_p>r\left(\delta\right)$;

\emph{(iii)} If $I_\delta\left(u\right)=0$ and ${\left\|\nabla u\right\|}_p\neq0$, then ${\left\|\nabla u\right\|}_p\geq r\left(\delta\right)$,\\
where $r\left(\delta\right)=\left(\frac\delta{C_\ast^{q+1}}\right)^\frac1{q+1-p}$, $C_\ast$ is the imbedding constant for $W^{1,p}\left(\Omega\right)\hookrightarrow L^{q+1}\left(\Omega\right)$ and satisfies
\begin{equation}\label{2.55}
\frac1{C_\ast}=\underset{u\in W_N^{1,p},u\neq0}{\inf}\frac{{\left\|\nabla u\right\|}_p}{{\left\|u\right\|}_{q+1}}.
\end{equation}
\end{lemma}
\begin{proof}
(i) By $0<{\left\|\nabla u\right\|}_p<r\left(\delta\right)$, we obtain
\begin{equation*}
\left\|u\right\|_{q+1}^{q+1}\leq C_\ast^{q+1}\left\|\nabla u\right\|_p^{q+1}=C_\ast^{q+1}\left\|\nabla u\right\|_p^{q+1-p}\left\|\nabla u\right\|_p^p<\delta\left\|\nabla u\right\|_p^p,
\end{equation*}
then $I_\delta\left(u\right)>0$.

(ii) From $I_\delta\left(u\right)<0$, we know that ${\left\|\nabla u\right\|}_p\neq0$. Then
\begin{equation*}
\delta\left\|\nabla u\right\|_p^p<\left\|u\right\|_{q+1}^{q+1}\leq C_\ast^{q+1}\left\|\nabla u\right\|_p^{q+1}=C_\ast^{q+1}\left\|\nabla u\right\|_p^{q+1-p}\left\|\nabla u\right\|_p^p
\end{equation*}
i.e., ${\left\|\nabla u\right\|}_p>r\left(\delta\right)$.

(iii) If $I_\delta\left(u\right)=0$, ${\left\|\nabla u\right\|}_p\neq0$,  we see that
\begin{equation*}
\delta\left\|\nabla u\right\|_p^p=\left\|u\right\|_{q+1}^{q+1}\leq C_\ast^{q+1}\left\|\nabla u\right\|_p^{q+1}=C_\ast^{q+1}\left\|\nabla u\right\|_p^{q+1-p}\left\|\nabla u\right\|_p^p,
\end{equation*}
which implies ${\left\|\nabla u\right\|}_p\geq r\left(\delta\right)$.
\end{proof}

\begin{lemma}[Properties of $J(\lambda u)$]\label{2.3}
Assume $u\in W_N^{1,p}\left(\Omega\right)$ and $\left\|u\right\|_{q+1}^{}\neq0$. Then

\emph{(i)} $J\left(\lambda u\right)$ is increasing on $\lambda\in\left[0,\lambda^\ast\right]$ and decreasing on $\lambda\in\lbrack\lambda^\ast,\infty)$, $\lambda=\lambda^\ast$ is the maximum point of $J\left(\lambda u\right)$. Furthermore, $\underset{\lambda\rightarrow0}{\mathrm{lim}}J\left(\lambda u\right)=0,\;\underset{\lambda\rightarrow+\infty}{\mathrm{lim}}J\left(\lambda u\right)=-\infty$;

\emph{(ii)} $I\left(\lambda u\right)>0$ for $0<\lambda<\lambda^\ast$ and $I\left(\lambda u\right)<0$ for $\lambda^\ast<\lambda<\infty$, and
$I\left(\lambda^\ast u\right)=0$, where
\begin{equation*}
  \lambda^\ast=\left(\frac{\left\|\nabla u\right\|_p^p}{\left\|u\right\|_{q+1}^{q+1}}\right)^\frac1{q-p+1}.
\end{equation*}
\end{lemma}
\begin{proof}
 From the definition of $J(u)$, we can find that
 \begin{equation*}
   J\left(\lambda u\right)=\frac{\lambda^p}p\left\|\nabla u\right\|_p^p-\frac{\lambda^{q+1}}{q+1}\left\|u\right\|_{q+1}^{q+1},
 \end{equation*}
which gives $\underset{\lambda\rightarrow0}{\mathrm{lim}}J\left(\lambda u\right)=0,\;\underset{\lambda\rightarrow+\infty}{\mathrm{lim}}J\left(\lambda u\right)=-\infty$. On the other hand, we see that
\begin{equation}\label{2.9}
\frac{dJ\left(\lambda u\right)}{d\lambda}=\lambda^{p-1}\left(\left\|\nabla u\right\|_p^p-\lambda^{q-p+1}\left\|u\right\|_{q+1}^{q+1}\right).
\end{equation}
Considering ${\left.\frac{dJ\left(\lambda u\right)}{d\lambda}\right|}_{\lambda=\lambda^\ast}=0$, we have
\begin{equation*}
  \lambda^\ast=\left(\frac{\left\|\nabla u\right\|_p^p}{\left\|u\right\|_{q+1}^{q+1}}\right)^\frac1{q-p+1}.
\end{equation*}
Taking into account (\ref{2.9}), we obtain
\begin{equation*}
\begin{array}{l}\frac{dJ\left(\lambda u\right)}{d\lambda}>0,\;\;\;\mathrm{for}\;0<\lambda<\lambda^\ast,\\\frac{dJ\left(\lambda u\right)}{d\lambda}<0,\;\;\;\mathrm{for}\;\lambda^\ast<\lambda<\infty.\end{array}
\end{equation*}
In addition,  it is easy to see that
\begin{equation*}
I\left(\lambda u\right)=\lambda\frac{dJ\left(\lambda u\right)}{d\lambda}=\left\{\begin{array}{l}>0,\;\;\;0<\lambda<\lambda^\ast,\\=0,\;\;\;\lambda=\lambda^\ast,\\<0,\;\;\;\lambda^\ast<\lambda<\infty.\end{array}\right.
\end{equation*}
\end{proof}

\begin{lemma}[Depth $d$ of potential well]\label{2.4}
Let $p$ and $q$ satisfy the conditions given by (\ref{1.1}). For the depth $d$ of potential well, we have
\begin{equation*}
d=\frac{q+1-p}{p\left(q+1\right)}C_*^{-\frac{p(q+1)}{q+1-p}}>0.
\end{equation*}
\end{lemma}
\begin{proof} If $u\in \mathcal N$, then ${\left\|\nabla u\right\|}_p\neq0$ and $I\left(u\right)=0$ (or $\left\|\nabla u\right\|_p^p-\left\|u\right\|_{q+1}^{q+1}=0$). It follows that
\begin{eqnarray*}\label{2.11}
J\left(u\right)=\frac1p\left\|\nabla u\right\|_p^p-\frac1{q+1}\left\|u\right\|_{q+1}^{q+1}=\left(\frac1p-\frac1{q+1}\right)\left\|\nabla u\right\|_p^p,
\end{eqnarray*}
which says $J\left(u\right)>0$.

From Lemma \ref{2.3} (ii), we see that $u\in\mathcal N$ must be of the form $u=\lambda^\ast v$ for $v\in W_N^{1,p}\left(\Omega\right)$, $v\neq 0$ and $\lambda^\ast=\left(\frac{\left\|\nabla v\right\|_p^p}{\left\|v\right\|_{q+1}^{q+1}}\right)^\frac1{q-p+1}$. Hence, we have
\begin{eqnarray*}
&&d=\inf_{u\in\mathcal N}J(u)=\inf_{u\in\mathcal N}\left[\left(\frac{1}{p}-\frac{1}{q+1}\right)\Vert\nabla u\Vert_p^p\right]\\
&&\quad =\left(\frac{1}{p}-\frac{1}{q+1}\right)\inf_{u\in\mathcal N}\left(\frac{\Vert\nabla u\Vert_p^{q+1}}{\Vert u\Vert_{q+1}^{q+1}}\right)^{\frac{p}{q+1-p}}\\
&&\quad =\left(\frac{1}{p}-\frac{1}{q+1}\right)\inf_{v\in W_N^{1,p}\left(\Omega\right), v\neq 0}\left(\frac{\Vert\nabla \left(\lambda^\ast v\right)\Vert_p^{q+1}}{\Vert \left(\lambda^\ast v\right)\Vert_{q+1}^{q+1}}\right)^{\frac{p}{q+1-p}}\\
&&\quad =\frac{q+1-p}{p\left(q+1\right)} C_*^{-\frac{p(q+1)}{q+1-p}}.
\end{eqnarray*}

The proof is complete.
\end{proof}

\begin{lemma}\label{2.5}
For $d(\delta)$ in (\ref{2.211}), we have

\emph{(i)} $d(\delta)\geq\frac1p(1-\delta)r^p(\delta)+\frac{q+1-p}{p\left(q+1\right)}\delta r^p(\delta)$. In particular, $d(1)\geq\frac{q+1-p}{p\left(q+1\right)}\left(\frac1{C_\ast^{q+1}}\right)^\frac p{q+1-p}$;

\emph{(ii)} there is a unique $b$, $b\in(1,\frac{q+1}p\rbrack$ such that $d(b)=0$, and $d(\delta)>0$ for $1\leq\delta<b$;

\emph{(iii)} $d(\delta)$ is increasing on $0<\delta\leq1$, decreasing on $1\leq\delta\leq b$, and $\delta=1$ is the maximum point of $d(\delta)$.
\end{lemma}
\begin{proof}
(i) If $u\in N_\delta$, that is, $I_\delta\left(u\right)=0$ and ${\left\|\nabla u\right\|}_p\neq0$. By Lemma \ref{2.2} (iii), we find
\begin{equation*}
\begin{array}{l}J\left(u\right)=\frac1p\left(1-\delta\right)\left\|\nabla u\right\|_p^p+\frac\delta p\left\|\nabla u\right\|_p^p-\frac1{q+1}\left\|u\right\|_{q+1}^{q+1}\\\;\;\;\;\;\;\;\;\;\geq\frac1p\left(1-\delta\right)r^p\left(\delta\right)+\frac{q+1-p}{p\left(q+1\right)}\delta r^p\left(\delta\right),\end{array}
\end{equation*}
which implies $d(1)\geq\frac{q+1-p}{p\left(q+1\right)}\left(\frac1{C_\ast^{q+1}}\right)^\frac p{q+1-p}$.

(ii) Let $\lambda\left(\delta\right)=\left(\frac{\delta\left\|\nabla u\right\|_p^p}{\left\|u\right\|_{q+1}^{q+1}}\right)^\frac1{q+1-p}$, then
\begin{equation*}
  \begin{array}{l}
I_\delta\left(\lambda\left(\delta\right)u\right)=\delta\lambda\left(\delta\right)^p\left\|\nabla u\right\|_p^p-\lambda\left(\delta\right)^{q+1}\left\|u\right\|_{q+1}^{q+1}\\\;\;\;\;\;\;\;\;\;\;\;\;\;\;\;\;\;=\lambda\left(\delta\right)^p\left(\delta\left\|\nabla u\right\|_p^p-\lambda\left(\delta\right)^{q+1-p}\left\|u\right\|_{q+1}^{q+1}\right)\\\;\;\;\;\;\;\;\;\;\;\;\;\;\;\;\;\;=0,
  \end{array}
\end{equation*}
i.e., $I_\delta\left(\lambda\left(\delta\right)u\right)=0$. For $\lambda\left(\delta\right)u\in N_\delta$, we have
\begin{equation*}
  \begin{array}{l}
d\left(\delta\right)\leq J\left(\lambda\left(\delta\right)u\right)\leq\frac1p\lambda\left(\delta\right)^p\left\|\nabla u\right\|_p^p-\frac1{q+1}\lambda\left(\delta\right)^{q+1}\left\|u\right\|_{q+1}^{q+1}\\\;\;\;\;\;\;\;\;\;\;\;\;\;\;\;\;\;\;\;\;\;\;\;\;\;\;\;\;=\left(\frac1p-\frac\delta{q+1}\right)\lambda\left(\delta\right)^{p(x)}\left\|\nabla u\right\|_p^p.
    \end{array}
\end{equation*}
Then $d\left(\delta\right)\leq0$ for $\delta=\frac{q+1}{p}$. On the other hand, $d\left(1\right)=d>0$. Since $d\left(\delta\right)$ is continuous with respect to $\delta$, there is a $b\in(1,\frac{q+1}{p}\rbrack$ satisfying $d\left(b\right)=0$.

(iii) If we can prove that $\underset{N_{\delta'}}{\inf}J\left(u\right)=d\left(\delta'\right)<d\left(\delta''\right)=\underset{N_{\delta''}}{\inf}J\left(u\right)$ for $0<\delta'<\delta''<1$ and $1<\delta''<\delta'<b$, respectively, then the lemma is proved.

In fact, for $0<\delta'<\delta''<1$, from the definition of $\lambda\left(\delta\right)$, we see that
\begin{equation}\label{2.99}
\begin{array}{l}J\left(\lambda\left(\delta''\right)u\right)-J\left(\lambda\left(\delta'\right)u\right)\\=\frac1p\int_\Omega(\lambda\left(\delta''\right)^p\left|\nabla u\right|^p-\lambda\left(\delta'\right)^p\left|\nabla u\right|^p)\operatorname dx\\\;\;\;\;-\frac1{q+1}\int_\Omega(\lambda\left(\delta''\right)^{q+1}\left|u\right|^{q+1}-\lambda\left(\delta'\right)^{q+1}\left|u\right|^{q+1})\operatorname dx\\=\int_\Omega\left|\nabla u\right|^p\int_{\lambda\left(\delta'\right)}^{\lambda\left(\delta''\right)}\lambda\left(\delta\right)^{p-1}\operatorname d\lambda\operatorname dx-\int_\Omega\left|u\right|^{q+1}\int_{\lambda\left(\delta'\right)}^{\lambda\left(\delta''\right)}\lambda\left(\delta\right)^q\operatorname d\lambda\operatorname dx\\=\int_{\lambda\left(\delta'\right)}^{\lambda\left(\delta''\right)}\left(\lambda\left(\delta\right)^{p-1}\left\|\nabla u\right\|_p^p-\lambda\left(\delta\right)^q\left\|u\right\|_{q+1}^{q+1}\right)\operatorname d\lambda\\=\int_{\lambda\left(\delta'\right)}^{\lambda\left(\delta''\right)}\lambda\left(\delta\right)^{p-1}\left(1-\delta\right)\left\|\nabla u\right\|_p^p\operatorname d\lambda.\end{array}
\end{equation}

By the definition of $\lambda\left(\delta\right)$, we know that $\lambda\left(\delta\right)$ is increasing with respect to $\delta$. Therefore, from (\ref{2.99}), we have $J\left(\lambda\left(\delta''\right)u\right)-J\left(\lambda\left(\delta'\right)u\right)>0$.

Similarly, for $1<\delta''<\delta'<b$, we get $J\left(\lambda\left(\delta''\right)u\right)-J\left(\lambda\left(\delta'\right)u\right)>0$.
\end{proof}

From Lemma \ref{2.5}, we can define $d_0=\underset{\delta\rightarrow0^+}{\mathrm{lim}}d(\delta)\geq0$.

\begin{lemma}\label{2.6}
For $u\in W_N^{1,p}$, let $d_0<J(u)<d$, and $\delta_1<1<\delta_2$ be the two roots of equation $d(\delta)=J(u)$, then the sign of $I_\delta(u)$ is invariable in $\delta_1<\delta<\delta_2$.
\end{lemma}

Next, we prove that $W_\delta$ and $V_\delta$ are invariant sets of (\ref{1.1}) when $0<J(u_0)<d$. The following discussion is divided into two parts: $J(u_0)$ being in the monotonic interval of $d\left(\delta\right)$ and $J(u_0)$ being in the non-monotonic interval of $d\left(\delta\right)$.

\begin{proposition}[Inspired by \cite{24}]\label{2.1}
 Assume that $u$ is a weak solution of (\ref{1.1}), and the initial value satisfies $u_0\in W_N^{1,p}$ and $J\left(u_0\right)=\sigma$. Then we get the following results.

 \emph{(i)} If $0<\sigma\leq d_0$, then there is a unique $\overline\delta\in\left(1,b\right)$ satisfying $d\left(\overline\delta\right)=\sigma$, where $b$ is the constant in Lemma \ref{2.5}  (ii). Further, if $I\left(u_0\right)>0$, then $u\in W_\delta$ for any $1\leq\delta<\overline\delta$. Otherwise, if $I\left(u_0\right)<0$, then there is $u\in V_\delta$ for any $1\leq\delta<\overline\delta$;

 \emph{(ii)} If $d_0<\sigma<d$, then $\delta_1$ and $\delta_2$ satisfy $\delta_1<1<\delta_2$ and $d\left(\delta_1\right)=d\left(\delta_2\right)=\sigma$. Further, if $I\left(u_0\right)>0$, then there is $u\in W_\delta$ for any  $\delta_1<\delta<\delta_2$. Otherwise, if $I\left(u_0\right)<0$, then $u\in V_\delta$ for any $\delta_1<\delta<\delta_2$.
\end{proposition}
\begin{proof}
Case 1. If $0<J\left(u_0\right)=\sigma\leq d_0$, that is, $J\left(u_0\right)$ is in the monotonic interval of $d\left(\delta\right)$. By Lemma \ref{2.5}, there is a unique $\overline\delta\in\left(1,b\right)$ satisfying $d\left(\overline\delta\right)=\sigma$. For any $1\leq\delta<\overline\delta$, we have
\begin{equation}\label{2.15}
I_\delta\left(u_0\right)=\left(\delta-1\right)\left\|\nabla u_0\right\|_p^p+I\left(u_0\right)\geq I\left(u_0\right),\;J\left(u_0\right)=\sigma=d\left(\overline\delta\right)<d\left(\delta\right).
\end{equation}

Multiplying the equation of (\ref{1.1}) by $u_t$, and
integrating on $\Omega\times\left[0,t\right]$, we obtain that
\begin{equation}\label{2.16}
\int_0^t\left(\left\|u_\tau\right\|_2^2+\left\|\nabla u_\tau\right\|_2^2\right)\operatorname d\tau+J\left(u\right)=J\left(u_0\right)=d\left(\overline\delta\right)<d\left(\delta\right),
\end{equation}
for any $t\in\left(0,T\right)$ and any $\delta\in\lbrack1,\overline\delta)$, where $T$ is the maximal existence time.

If $I\left(u_0\right)>0$, it is easy to know from (\ref{2.15}) that there is $u_0\in W_\delta$ for $\delta\in\lbrack1,\overline\delta)$. We assert that if $t\in\left(0,T\right)$ and $\delta\in\lbrack1,\overline\delta)$, then $u\in W_\delta$. Assuming that the assertion is not tenable, then there exists $\delta^\ast\in\lbrack1,\overline\delta)$ and $t_0\in\left(0,T\right)$, such that $u\in W_{\delta^\ast}$ for $t\in\left(0,t_0\right)$, but $u\left(x,t_0\right)\in\partial W_{\delta^\ast}$, i.e.,
\begin{equation*}
I_{\delta^\ast}\left(u\left(t_0\right)\right)=0,\;\left\|\nabla u\left(t_0\right)\right\|_p^p\neq0,\;\mathrm{or}\;J\left(u\left(t_0\right)\right)=d\left(\delta^\ast\right).
\end{equation*}

In fact, from (\ref{2.16}), we can see that $J\left(u\left(t_0\right)\right)\leq J\left(u_0\right)<d\left(\delta^\ast\right)$,  which implies
$I_{\delta^\ast}\left(u\left(t_0\right)\right)=0$ and $\left\|\nabla u\left(t_0\right)\right\|_p^p\neq0$, i.e. $u\left(x,t_0\right)\in N_{\delta^\ast}$. Therefore, by the definition of $d\left(\delta^\ast\right)$, we can get $J\left(u\left(t_0\right)\right)\geq d\left(\delta^\ast\right)$, which is a contradiction.

Next, we prove that if $I\left(u_0\right)<0$, then $u_0\in V_\delta$ for $\delta\in\lbrack1,\overline\delta)$, and $u\in V_\delta$ for any $t\in\left(0,T\right)$ and any $\delta\in\lbrack1,\overline\delta)$. Assuming that the assertion about $u_0$ is not tenable, combined with (\ref{2.15}),  we can know that $\delta_\ast\in\lbrack1,\overline\delta)$ being the first number such that $u_0\in V_\delta$ for $\delta\in\lbrack1,\delta_\ast)$ and $u_0\in\partial V_{\delta_\ast}$, i.e.,
\begin{equation*}
I_{\delta_\ast}\left(u_0\right)=0,\;or\;J\left(u_0\right)=d\left(\delta_\ast\right).
\end{equation*}

Because $J(u_0)$ is in the strict decreasing interval of $d(\delta)$, then $J(u_0)=d(\overline\delta)<d(\delta_\ast)$, which implies that $I_{\delta_\ast}(u_0)=0$. Because $I_\delta(u_0)<0$ for $\delta\in\lbrack1,\delta_\ast)$,  then from Lemma \ref{2.2} ($\mathrm{ii}$) we have ${\left\|\nabla u_0\right\|}_p>r(\delta)>0$, that is, $u_0\in N_{\delta_\ast}$. According to the definition of $d\left(\delta_\ast\right)$, we get $J(u_0)=d(\overline\delta)\geq d(\delta_\ast)$, which is contradict with the monotonicity of $d(\delta)$. Assuming that the assertion about $u$ is not tenable,  then for $t\in(0,t_0)$, there exists $\delta_\ast^\ast\in\lbrack1,\overline\delta)$ and $t_0\in(0,T)$, such that $u\in V_{\delta_\ast^\ast}$, but $u(x,t_0)\in\partial V_{\delta_\ast^\ast}$, that is,
\begin{equation*}
I_{\delta_\ast^\ast}(u(t_0))=0,\;\mathrm{or}\;J(u(t_0))=d(\delta_\ast^\ast).
\end{equation*}

In fact, $J(u(t_0))\leq J(u_0)<d(\delta_\ast^\ast)$ is obtained from (\ref{2.16}),  which implies $I_{\delta_\ast^\ast}(u(t_0))=0$. Because $I_\delta(u(t_0))<0$ for $\delta\in\lbrack1,\delta_\ast^\ast)$,  then from Lemma \ref{2.2} ($\mathrm{ii}$) we get ${\left\|\nabla u(t_0)\right\|}_p>r(\delta)>0$, that is, $u(x,t_0)\in N_{\delta_\ast^\ast}$. From the definition of $d(\delta_\ast^\ast)$, we get $J(u(t_0))\geq d(\delta_\ast^\ast)$, which is a contradiction.

Case 2. $d_0<J(u_0)=\sigma<d$, that is, $J(u_0)$ is in the non-monotonic interval of $d(\delta)$. By Lemma \ref{2.5}, it is known that there exist $\delta_1<1<\delta_2$ being two roots of $d(\delta)=\sigma$, and $d_0<J(u_0)=d(\delta_1)=d(\delta_2)<d(\delta)$ for $\delta\in(\delta_1,\delta_2)$. If $I(u_0)>0$, then from Lemma \ref{2.6},  the sign of $I_\delta(u)$ remains unchanged for $\delta_1<\delta<\delta_2$.  Thus $I_\delta(u_0)>0$ for $\delta\in(\delta_1,\delta_2)$. Therefore, we have $u_0\in W_\delta$ for $\delta\in(\delta_1,\delta_2)$. The proof of $u\in W_\delta$ is similar to Case 1. If $I(u_0)<0$, by lemma \ref{2.6}, we can still get $I_\delta(u_0)<0$ for $\delta\in(\delta_1,\delta_2)$ and $J(u_0)<d(\delta)$ for $\delta\in(\delta_1,\delta_2)$,  imply that $u_0\in V_\delta$ for $\delta\in(\delta_1,\delta_2)$. The proof of $u\in V_\delta$ is similar to Case 1.
\end{proof}

By Lagrange mean value theorem, the following lemma can be obtained directly.

\begin{lemma}[Estimate of nonlinear term $\left|u\right|^{q-1}u$]\label{2.7}
If $q$ satisfies the conditions given by (\ref{1.1}), $\left|u_1\right|+\left|u_2\right|>0$ and $u_1\neq u_2$ for any $u_1\left(x,t\right)$, $u_2\left(x,t\right)$ with $\left(x,t\right)\in\Omega\times\left[0,T\right]$, then
\begin{equation*}
\left|u_1\right|^{q-1}u_1-\left|u_2\right|^{q-1}u_2\leq q\left(\left|u_1\right|+\left|u_2\right|\right)^{q-1}\left|u_1-u_2\right|.
\end{equation*}
\end{lemma}

The main results of global existence, uniqueness, decay estimate and blowup for the case $J\left(u_0\right)<d$ are stated in the following two theorems.
\begin{theorem}\label{theorem1}
 Assume that $u_0\in W_N^{1,p}\left(\Omega\right)$, $J\left(u_0\right)<d$  and $I\left(u_0\right)>0$, then problem (\ref{1.1}) has a global weak solution $u\in L^\infty\left(0,\infty;W_N^{1,p}\left(\Omega\right)\right)$ with $u_t\in L^2\left(0,\infty;W_N^{1,2}\left(\Omega\right)\right)$, and the weak solution of problem (\ref{1.1}) is unique for $p\leq2$.

 Moreover, if $\left\|\nabla u\right\|_2^2\leq C\left\|\nabla u\right\|_p^p$, there exists a constant $\delta>0$ such that $\left\|u\right\|_{H^1}^2<\left\|u_0\right\|_{H^1}^2e^{-2\delta t}$ for $t\in\lbrack0,\infty)$.
\end{theorem}

\begin{theorem}\label{theorem2}
Let $u_0\in W_N^{1,p}\left(\Omega\right)$, $J\left(u_0\right)<d$ and $I\left(u_0\right)<0$, then the weak solution $u(t)$ of problem (\ref{1.1}) blows up in finite time.
\end{theorem}

The main results for the case $J\left(u_0\right)=d$ are given in the following two theorems.

\begin{theorem}\label{theorem3}
 Assume that $u_0\in W_N^{1,p}(\Omega)$ and $J\left(u_0\right)=d$. If $I(u_0)\geq0$, then problem (\ref{1.1}) admits a global weak solution  $u(t)\in L^\infty(0,\infty; W_N^{1,p}(\Omega))$ with $u_t\in L^2\left(0,\infty;W_N^{1,2}\left(\Omega\right)\right)$, and there exists a unique weak solution for $p\leq2$.

 Moreover, if $I(u_0)>0$ and $\left\|\nabla u\right\|_2^2\leq C\left\|\nabla u\right\|_p^p$, there exist constants $t_1>0$ and $\kappa>0$ such that $\left\|u\right\|_{H^1}^2<\left\|u\left(t_1\right)\right\|_{H^1}^2e^{-2\kappa\left(t-t_1\right)}$ for $t\in(t_1,+\infty)$.
\end{theorem}

\begin{theorem}\label{theorem4}
Let $u_0\in W_N^{1,p}(\Omega)$, $J(u_0)=d$ and $I(u_0)<0$, then the weak solution $u(t)$ of problem (\ref{1.1}) blows up in finite time.
\end{theorem}

The following theorem gives the main result for case $J\left(u_0\right)>d$.

\begin{theorem}\label{theorem5}
 Assume that $u_0\in W_N^{1,p}\left(\Omega\right)$, $J\left(u_0\right)$ is finite and $J\left(u_0\right)>d$, $I\left(u_0\right)>0$ and $\left\|u_0\right\|_{H^1}^2\leq\lambda_{J\left(u_0\right)}$, then problem (\ref{1.1}) has a global weak solution $u\in L^\infty\left(0,\infty; W_N^{1,p}(\Omega)\right)$ with $u_t\in L^2\left(0,\infty;W_N^{1,2}(\Omega)\right)$, and the weak solution is unique for $p\leq2$.

  Moreover, we have that $u\left(t\right)\rightarrow0$ as $t\rightarrow+\infty$.
\end{theorem}

\section{Subcritical initial energy $\boldsymbol J\mathbf{(u_0)}\boldsymbol<\boldsymbol d$}

In this section, we shall consider the global existence, uniqueness and decay estimate, and blowup of solutions to problem (\ref{1.1}) for the subcritical initial energy $J\left(u_0\right)<d$.

\textbf{Proof of Theorem \ref{theorem1}.}

\textbf{Global existence}.  Consider a orthonormal basis $\left\{\omega_j\left(x\right)\right\}$ on $W_N^{1,p}\left(\Omega\right)$,  which is also orthogonal in $L^2\left(\Omega\right)$. Define
\begin{equation*}
u_m\left(x,t\right)=\sum_{j=1}^mg_{jm}\left(t\right)\omega_j\left(x\right),\;\;\;m=1,2,\cdots,
\end{equation*}
where $g_{jm}\left(t\right)$ satisfies the initial value problem of the following equations
\begin{equation}\label{3.2}
\left(u_{mt},\omega_s\right)+\left(\nabla u_{mt},\nabla\omega_s\right)+\left(\left|\nabla u_m\right|^{p-2}\nabla u_m,\nabla\omega_s\right)=\left(f\left(u_m\right),\omega_s\right),
\end{equation}
\begin{equation}\label{3.3}
u_m\left(x,0\right)=\sum_{j=1}^ma_{jm}\omega_j\left(x\right)\rightarrow u_0\left(x\right)\quad{\rm in}\quad W_N^{1,p}\left(\Omega\right),
\end{equation}
for $1\leq s\leq m$, in which
\begin{equation*}
a_{jm}=g_{jm}{(0)}\;\mathrm{and}\;f\left(u\right)=\left|u\right|^{q-1}u-\bbint_\Omega\left|u\right|^{q-1}udx.
\end{equation*}

Multiplying the $j^{th}$ equation of (\ref{3.2}) by $g'_{sm}\left(t\right)$, summing up with respect to $j$, and
integrating with respect to the time variable from 0 to $t$, we readily get the equality
\begin{equation}\label{3.4}
\int_0^t\left\|u_{m\tau}\right\|_{H^1}^2\operatorname d\tau+J\left(u_m\right)=J\left(u_m\left(0\right)\right),\;\;\;0\leq t<\infty.
\end{equation}

From (\ref{3.3}), we obtain $J\left(u_m\left(0\right)\right)\rightarrow J\left(u_0\right)<d$ and $I\left(u_m\left(0\right)\right)\rightarrow I\left(u_0\right)>0$, which combining (\ref{3.4}) implies that
\begin{equation}\label{3.5}
\int_0^t\left\|u_{m\tau}\right\|_{H^1}^2\operatorname d\tau+J\left(u_m\right)<d,\;\;\;0\leq t<\infty.
\end{equation}

In fact, we obtain $u_m\left(t\right)\in W$ for sufficiently large $m$. If the above is false, then  there exists $t_0>0$ to satisfy $u_m(x,t_0)\in\partial W$,  i.e.
\begin{equation*}
I(u_m(t_0))=0,\;{\left\|\nabla u_m(t_0)\right\|}_p\neq0,\;\mathrm{or}\;J(u_m(t_0))=d.
\end{equation*}

On the one hand, from (\ref{3.5}), it is easy to know that $J(u_m(t_0))=d$ is not tenable. On the other hand, if $I(u_m(t_0))=0$ and ${\left\|\nabla u_m(t_0)\right\|}_p\neq0$, it can be seen from the definition of $d$ that $J(u(t_0))\geq d$, which is also contradict with (\ref{3.5}). Thus we get
 $u_m\left(t\right)\in W$.

 With the help of (\ref{3.5}) and
\begin{equation*}
J\left(u_m\right)=\frac1{q+1}I\left(u_m\right)+\left(\frac1p-\frac1{q+1}\right)\left\|\nabla u_m\right\|_p^p,
\end{equation*}
we have
\begin{equation*}
\int_0^t\left\|u_{m\tau}\right\|_{H^1}^2\operatorname d\tau+\left(\frac1p-\frac1{q+1}\right)\left\|\nabla u_m\right\|_p^p<d,\;\;\;0\leq t<\infty,
\end{equation*}
where $m$ is sufficiently large, which yields a priori estimate
\begin{equation*}
\int_0^t\left\|u_{m\tau}\right\|_{H^1}^2\operatorname d\tau<d,\;\;\;0\leq t<\infty,
\end{equation*}
\begin{equation*}
\left\|\nabla u_m\right\|_p^p<\frac{p\left(q+1\right)}{q+1-p}d,\;\;\;0\leq t<\infty,
\end{equation*}
\begin{equation*}
\left\|u_{{}_m}^q\right\|_\gamma^\gamma=\left\|u_m\right\|_{q+1}^{q+1}\leq C_\ast^{q+1}\left(\frac{p\left(q+1\right)}{q+1-p}d\right)^\frac{q+1}p,\;\gamma=\frac{q+1}q,\;\;0\leq t<\infty.
\end{equation*}
Then, there exists a subsequence of $\left\{u_m\right\}$, denoted by the same symbol satisfying
\begin{equation*}
\left\{\begin{array}{l}u_m\rightarrow u\;\;\mathrm{in}\;L^\infty(0,\infty;\;W_N^{1,p}(\Omega))\;\mathrm{weak}\;\mathrm{star}\;\mathrm{and}\;\mathrm a.\mathrm e.\;\mathrm{in}\;Q=\Omega\times(0,\;\infty),\\u_{mt}\rightarrow u_t\;\;\mathrm{in}\;L^2(0,\infty;\;W_N^{1,2}(\Omega))\;\mathrm{weak},\\\left|\nabla u_m\right|^{p-2}\nabla u_m\rightarrow\chi\;\;\mathrm{in}\;L^\infty(0,\infty;\;L^\frac p{p-1}(\Omega))\;\mathrm{weak}\;\mathrm{star},\\\left|u_m\right|^{q-1}u_m\rightarrow\left|u\right|^{q-1}u\;\;\mathrm{in}\;L^\infty(0,\infty;\;L^\frac{q+1}q(\Omega))\;\mathrm{weak}\;\mathrm{star}.\end{array}\right.
\end{equation*}

We fixed $s$ in (\ref{3.2}) and letting $m\rightarrow\infty$,  there hold
\begin{equation*}
\left(u_t,\omega_s\right)+\left(\nabla u_t,\nabla\omega_s\right)+\left(\chi,\nabla\omega_s\right)=\left(f\left(u\right),\omega_s\right)\;\mathrm{for}\;\mathrm{all}\;s,
\end{equation*}
and
\begin{equation}\label{3.1666}
\left(u_t,v\right)+\left(\nabla u_t,\nabla v\right)+\left(\chi,\nabla v\right)=\left(f\left(u\right),v\right)\;\mathrm{for}\;\mathrm{all}\;v\in W_N^{1,p}\left(\Omega\right),\;t\in\left(0,\;\infty\right).
\end{equation}

Next, we only need to prove that there is $\chi=\left|\nabla u\right|^{p-2}\nabla u$ in the weak sense, i.e.
\begin{equation}\label{3.16666}
\left(\chi,\nabla v\right)=\left(\left|\nabla u\right|^{p-2}\nabla u,\nabla v\right),\;\;\;\forall v\in W_N^{1,p}\left(\Omega\right).
\end{equation}

In fact, for any $\psi\in L^\infty(0,T;W_N^{1,p}\left(\Omega\right)),\;\varsigma\in W_N^{1,p}\left(\Omega\right),\;0\;\leq\varsigma\leq1$, we obtain
\begin{equation*}
  \int_\Omega\varsigma\left(\left|\nabla u_m\right|^{p-2}\nabla u_m-\left|\nabla\psi\right|^{p-2}\nabla\psi\right)\nabla\left(u_m-\psi\right)\operatorname dx\geq0,
\end{equation*}
i.e.
\begin{equation}\label{3.177}
  \begin{array}{l}\int_\Omega\varsigma\left|\nabla u_m\right|^{p-2}\left|\nabla u_m\right|^2\operatorname dx-\int_\Omega\varsigma\left|\nabla u_m\right|^{p-2}\nabla u_m\nabla\psi\operatorname dx\\\;\;\;\;\;\;\;\;\;\;\;\;\;\;\;\;\;\;\;\;\;\;\;\;\;\;\;\;\;\;\;\;\;\;\;\;\;\;\;\;\;\;-\int_\Omega\varsigma\left|\nabla\psi\right|^{p-2}\nabla\psi\nabla\left(u_m-\psi\right)\operatorname dx\geq0.\end{array}
\end{equation}

Noticing that
\begin{equation}\label{3.1777}
\begin{array}{l}\int_\Omega\varsigma\left|\nabla u_m\right|^{p-2}\left|\nabla u_m\right|^2\operatorname dx\\=-\int_\Omega\nabla\varsigma\left|\nabla u_m\right|^{p-2}\nabla u_mu_m\operatorname dx-\int_\Omega\varsigma\nabla\left(\left|\nabla u_m\right|^{p-2}\nabla u_m\right)u_m\operatorname dx\\=-\int_\Omega\nabla\varsigma\left|\nabla u_m\right|^{p-2}\nabla u_mu_m\operatorname dx-\int_\Omega\varsigma u_{mt}u_m\operatorname dx-\int_\Omega\varsigma\nabla u_{mt}\nabla u_m\operatorname dx\\\;\;\;\;-\int_\Omega\nabla\varsigma\nabla u_{mt}u_m\operatorname dx+\int_\Omega\varsigma\left|u\right|_m^{q-1}u_m^{}u_m\operatorname dx,\end{array}
\end{equation}

Letting $m\rightarrow\infty$ in (\ref{3.177}), combined with (\ref{3.1777}),we find
\begin{equation}\label{3.17777}
\begin{array}{l}-\int_\Omega\nabla\varsigma\chi u\operatorname dx-\int_\Omega\varsigma u_tu\operatorname dx-\int_\Omega\varsigma\nabla u_t\nabla u\operatorname dx-\int_\Omega\nabla\varsigma\nabla u_tu\operatorname dx\\+\int_\Omega\varsigma\left|u\right|_{}^{q-1}uu\operatorname dx-\int_\Omega\varsigma\chi\nabla\psi\operatorname dx-\int_\Omega\varsigma\left|\nabla\psi\right|^{p-2}\nabla\psi\nabla\left(u-\psi\right)\operatorname dx\geq0.\end{array}
\end{equation}

Choosing $v=u\varsigma$ in (\ref{3.1666}), we have
\begin{equation}\label{3.177777}
   \begin{array}{l} \int_\Omega u_tu\varsigma\operatorname dx+\int_\Omega\nabla u_t\nabla u\varsigma\operatorname dx+\int_\Omega\nabla u_tu\nabla\varsigma\operatorname dx+\int_\Omega\chi\nabla u\varsigma\operatorname dx+\int_\Omega\chi u\nabla\varsigma\operatorname dx\\=\int_\Omega\left|u\right|_{}^{q-1}uu\varsigma\operatorname dx.\end{array}
\end{equation}

Combining (\ref{3.17777}) with (\ref{3.177777}), we get
\begin{equation*}
  \int_\Omega\varsigma\left(\chi-\left|\nabla\psi\right|^{p-2}\nabla\psi\right)\nabla\left(u-\psi\right)\operatorname dx\geq0.
\end{equation*}
Taking $\psi=u-\lambda v,\;\lambda\geq0,\;v\in W_N^{1,p}\left(\Omega\right)$ in the above inequality, we obtain
\begin{equation}\label{3.1888}
  \int_\Omega\varsigma\left(\chi-\left|\nabla\left(u-\lambda v\right)\right|^{p-2}\nabla\left(u-\lambda v\right)\right)\nabla v\operatorname dx\geq0.
\end{equation}
Choosing $\lambda\rightarrow0$ in (\ref{3.1888}), we get
\begin{equation*}
  \int_\Omega\varsigma\left(\chi-\left|\nabla u\right|^{p-2}\nabla u\right)\nabla v\operatorname dx\geq0,\;\forall v\in W_N^{1,p}\left(\Omega\right).
\end{equation*}
Obviously, if we choose $\lambda\leq0$, we can deduce the similar inequality replacing $\geq$ by $\leq$. Hence, (\ref{3.16666}) holds. Furthermore, (\ref{3.3}) gives $u(x,0)=u_0(x)$ in $W_N^{1,p}(\Omega)$.

\textbf{Uniqueness}. Let $u_1$ and $u_2$ be the weak solutions of (\ref{1.1}) with the same initial data. Define $v=u_1-u_2$, we have
\begin{equation}\label{3.17}
v_\tau\;+\nabla v_\tau+\mathrm{div}\left(\vert\nabla u_1\vert^{p-2}\nabla u_1\right)-\mathrm{div}\left(\vert\nabla u_2\vert^{p-2}\nabla u_2\right)=\vert u_1\vert^{q-1}u_1-\vert u_2\vert^{q-1}u_2,
\end{equation}
\begin{equation*}
u_1,u_2,v\in L^\infty\left(0,\infty;W_N^{1,p}\right),\;u_{1t},u_{2t},v_t\in L^2\left(0,\infty;W_N^{1,2}\right).
\end{equation*}

Multiplying (\ref{3.17}) by $v$, integrating over $\left(0,t\right)\times\Omega$, we get
\begin{equation}\label{3.19}
\begin{array}{l}\int_0^t\int_\Omega^{}{(v_\tau\;v+\nabla v_\tau\nabla v+\left(\vert\nabla u_1\vert^{p-2}\nabla u_1-\vert\nabla u_2\vert^{p-2}\nabla u_2\right)\nabla v\;)}\operatorname dx\operatorname d\tau\\=\int_0^t\int_\Omega^{}\left(\vert u_1\vert^{q-1}u_1-\vert u_2\vert^{q-1}u_2\right)v\;\operatorname dx\operatorname d\tau.\end{array}
\end{equation}

By Lemma \ref{2.7}, we get
\begin{equation}\label{3.20}
\begin{array}{l}\int_0^t\int_\Omega\left(\left|u_1\right|^{q-1}u_1-\left|u_2\right|^{q-1}u_2\right)v\operatorname dx\operatorname d\tau\\\leq\int_0^t\int_\Omega q\left(\left|u_1\right|+\left|u_2\right|\right)^{q-1}\left|u_1-u_2\right|\left|v\right|\operatorname dx\operatorname d\tau\\\leq C\int_0^t{\left\|\left(\left|u_1\right|+\left|u_2\right|\right)^{q-1}\right\|}_{A_1}{\left\|u_1-u_2\right\|}_{A_2}{\left\|v\right\|}_{A_3}\operatorname d\tau\\=C\int_0^t\left\|\left|u_1\right|+\left|u_2\right|\right\|_{{}^{\left(q-1\right)}A_1}^{q-1}{\left\|u_1-u_2\right\|}_{A_2}{\left\|v\right\|}_{A_3}\operatorname d\tau\\\leq C\int_0^t{\left\|u_1-u_2\right\|}_{H^1}{\left\|v\right\|}_{H^1}\operatorname d\tau\\\leq C\int_0^t\left\|v\right\|_{H^1}^2\operatorname d\tau,\end{array}
\end{equation}
where $(q-1)A_1=q+1$ and $A_2=A_3=q+1<\frac{2n}{n-2}$ by $p\leq2$. As $v(x,0)=0$, we obtain
\begin{equation}\label{3.1995}
\begin{array}{l}\int_0^t\int_\Omega{(v_\tau v+\nabla v_\tau\nabla v)}\operatorname dx\operatorname d\tau=\frac1{\;2}\int_\Omega{(\left.v^2\left(\tau\right)\right|_{v\left(0\right)}^{v\left(t\right)}+\left.\left|\nabla v\left(\tau\right)\right|^2\right|_{v\left(0\right)}^{v\left(t\right)})}\operatorname dx\\\;\;\;\;\;\;\;\;\;\;\;\;\;\;\;\;\;\;\;\;\;\;\;\;\;\;\;\;\;\;\;\;\;\;\;\;\;\;\;\;\;\;\;=\frac12\int_\Omega{(v^2\left(t\right)+\left|\nabla v\left(t\right)\right|^2)}\operatorname dx\\\;\;\;\;\;\;\;\;\;\;\;\;\;\;\;\;\;\;\;\;\;\;\;\;\;\;\;\;\;\;\;\;\;\;\;\;\;\;\;\;\;\;\;=\frac12\left\|v\left(t\right)\right\|_{H^1}^2.\end{array}
\end{equation}
Let $\overset-v=\theta_1u_1+\left(1-\theta_1\right)u_2,\;\theta_1\in\left[0,1\right],$ i.e.,
\begin{equation}\label{3.21}
\int_\Omega\left(\left|\nabla u_1\right|^{p-2}\nabla u_1-\left|\nabla u_2\right|^{p-2}\nabla u_2\right)\nabla v\operatorname dx=\int_\Omega\left(p-1\right)\left|\nabla\overset-v\right|^{p-2}\left|\nabla v\right|^2\operatorname dx.
\end{equation}
It follows from (\ref{3.19}), (\ref{3.20}), (\ref{3.1995}) and (\ref{3.21}) that
\begin{equation*}
\left\|v\right\|_{H^1}^2\leq C\int_0^t\left\|v\right\|_{H^1}^2\operatorname d\tau.
\end{equation*}
By Gronwall's inequality, we obtain $\left\|v\right\|_{H^1}^2\leq0$, that is, $\left\|v\right\|_{H^1}^2=\left\|u_1-u_2\right\|_{H^1}^2=0$. Thus $u_1=u_2=0$ a.e. in $\Omega\times\left(0,\infty\right)$.

\textbf{Asymptotic behavior}. Letting $\varphi=u$ in (\ref{2.331}) and
\begin{equation*}
  \int_0^t\int_\Omega\left(\bbint_\Omega^{}\left|u\right|^{q-1}u\operatorname dx\right)u\operatorname dx\operatorname d\tau=\int_0^t\left(\bbint_\Omega^{}\left|u\right|^{q-1}u\operatorname dx\right)\int_\Omega u\operatorname dx\operatorname d\tau=0,
\end{equation*}
we find
\begin{equation*}
  \frac12\left\|u\right\|_{H^1}^2-\frac12\left\|u_0\right\|_{H^1}^2+\;\int_0^t\left(\left\|\nabla u\right\|_p^p-\left\|u\right\|_{q+1}^{q+1}\right)\operatorname ds=0,
\end{equation*}
then
\begin{equation}\label{3.22}
  \frac12\frac d{dt}\left\|u\right\|_{H^1}^2=-I\left(u\right).
\end{equation}

By Proposition \ref{2.1}, we know that $u\in W_\delta$ for $1\leq\delta<\overline\delta$ or $\delta_1<\delta<\delta_2$ with $\delta_1<1<\delta_2$,
and particularly $I(u)>0$. Then, the norm ${\left\|u\right\|}_{H^1}$ is equivalent to the norm ${\left\|\nabla u\right\|}_2$ on $H^1\left(\Omega\right)$ and combining (\ref{2.55}), we get
\begin{equation}\label{3.23}
\begin{array}{l}\;\;\frac12\frac d{dt}\left\|u\right\|_{H^1}^2=-I\left(u\right)=-\left\|\nabla u\right\|_p^p+\left\|u\right\|_{q+1}^{q+1}\\\;\;\;\;\;\;\;\;\;\;\;\;\;\;\;\;\;\;\;\leq-\left\|\nabla u\right\|_p^p+C_\ast^{q+1}\left\|\nabla u\right\|_p^{q+1}\\\;\;\;\;\;\;\;\;\;\;\;\;\;\;\;\;\;\;\;=\left(C_\ast^{q+1}\left\|\nabla u\right\|_p^{q-p+1}-1\right)\left\|\nabla u\right\|_p^p\\\;\;\;\;\;\;\;\;\;\;\;\;\;\;\;\;\;\;\;\leq C\left(C_\ast^{q+1}\left\|\nabla u\right\|_p^{q-p+1}-1\right)\left\|\nabla u\right\|_2^2\\\;\;\;\;\;\;\;\;\;\;\;\;\;\;\;\;\;\;\;\leq C\left(C_\ast^{q+1}\left\|\nabla u\right\|_p^{q-p+1}-1\right)\left\|u\right\|_{H^1}^2.\end{array}
\end{equation}

Now, we will discuss $C_\ast^{q+1}\left\|\nabla u\right\|_p^{q-p+1}-1<0$. Taking into account (\ref{2.1}) and $I(u)>0$ we deduce that
\begin{equation*}
  J(u)>\left(\frac1p-\frac1{q+1}\right)\;\left\|\nabla u\right\|_p^p,
\end{equation*}
that is,
\begin{equation}\label{3.24}
\left\|\nabla u\right\|_p^p<\frac{p\left(q+1\right)}{q+1-p}J\left(u\right)\leq\frac{p\left(q+1\right)}{q+1-p}J\left(u_0\right).
\end{equation}

By (\ref{3.24}),  we see that
\begin{equation}\label{3.25}
C_\ast^{q+1}\left\|\nabla u\right\|_p^{q+1-p}<C_\ast^{q+1}\left(\frac{p\left(q+1\right)}{q+1-p}J\left(u_0\right)\right)^\frac{q-p+1}p<1.
\end{equation}

From (\ref{3.23}) it follows that
\begin{equation*}
\frac12\frac d{dt}\left\|u\right\|_{H^1}^2<C\left(C_\ast^{q+1}\left(\frac{p\left(q+1\right)}{q+1-p}J\left(u_0\right)\right)^\frac{q-p+1}p-1\right)\left\|u\right\|_{H^1}^2.
\end{equation*}

Considering (\ref{3.25}), we can know that there exists $\delta>0$ such that
\begin{equation*}\label{3.27}
-\delta:=C\left(C_\ast^{q+1}\left(\frac{p\left(q+1\right)}{q+1-p}J\left(u_0\right)\right)^\frac{q-p+1}p-1\right).
\end{equation*}
Applying Gronwall's inequality, we infer for $\delta>0$
\begin{equation*}\label{3.28}
\left\|u\right\|_{H^1}^2<\left\|u_0\right\|_{H^1}^2e^{-2\delta t},\;0\leq t<\infty.
\end{equation*}
Theorem \ref{theorem1} is proved.

Next, in order to prove $u(t)$ blowing up in finite time, we give a relationship between $\left\|\nabla u\right\|_p^p$ and the depth $d$ of potential well with $u_0\in V$.

\begin{lemma}\label{3.1}
Let $u_0\in V$, then
\begin{equation*}\label{3.29}
\left\|\nabla u\right\|_p^p>\frac{p\left(q+1\right)}{q+1-p}d.
\end{equation*}
\end{lemma}
\begin{proof}
Assume that $u(t)$ is a weak solution of (\ref{1.1}) with $J(u_0)<d$ and $I(u_0)<0$, $T$ is the maximal existence time. Recalling Proposition \ref{2.1},  it is easy to see that $u(x,t)\in V_\delta$, that is, $I(u)<0$ for $0<t<T$. From Lemma \ref{2.2}, take $\delta=1$, we have
\begin{equation*}
\left\|\nabla u\right\|_p^p>\left(\frac1{C_\ast^{q+1}}\right)^\frac p{q+1-p}.
\end{equation*}
Coming back to Lemma \ref{2.4}, we conclude that
\begin{equation*}\label{3.30}
d<\frac{q+1-p}{p\left(q+1\right)}\left\|\nabla u\right\|_p^p.
\end{equation*}
\end{proof}
\textbf{Proof of Theorem \ref{theorem2}.} We will prove $u(t)$ blowing up in finite time with $u_0\in V$. Arguing by contradiction, assume that the solution is global in time.

Taking into account $u_0\in V$ and Proposition \ref{2.1}, we obtain $u\in V_\delta$ for $t\in [0, +\infty)$. Denoting
\begin{equation*}\label{3.31}
G\left(t\right)=\int_0^t\left\|u\right\|_{H^1}^2\operatorname d\tau+\left(T_0-t\right)\left\|u_0\right\|_{H^1}^2,\;t\in\left[0,T_0\right],
\end{equation*}
where $0<T_0<+\infty$. Obviously, for any $t\in [0,T_0]$ we get $G(t)>0$. Applying the continuity of $G(t)$ with respect to $t$, we infer that there exists a constant $\theta>0$, for $t\in [0,T_0]$ such that $G(t)\geq\theta$. Then
\begin{equation}\label{3.32}
\begin{array}{l}G'\left(t\right)=\left\|u\right\|_{H^1}^2-\left\|u_0\right\|_{H^1}^2\\\;\;\;\;\;\;\;\;\;=\left\|u\right\|_2^2+\left\|\nabla u\right\|_2^2-\left\|u_0\right\|_2^2-\left\|\nabla u_0\right\|_2^2\\\;\;\;\;\;\;\;\;\;=2\int_0^t\left(u,u_\tau\right)\operatorname d\tau+2\int_0^t\left(\nabla u,\nabla u_\tau\right)\operatorname d\tau,\end{array}
\end{equation}
and (\ref{3.22}) gives
\begin{equation}\label{3.33}
G''\left(t\right)=2\left(u,u_t\right)+2\left(\nabla u,\nabla u_t\right)=-2I\left(u\right).
\end{equation}
By (\ref{3.32}) and Cauchy-Schwarz inequality,  we deduce
\begin{equation}\label{3.34}
\begin{array}{l}\left(G'\left(t\right)\right)^2=4\left(\int_0^t{\left(u,u_\tau\right)}_{}\operatorname d\tau+\int_0^t{\left(\nabla u,\nabla u_\tau\right)}_{}\operatorname d\tau\right)^2\\\;\;\;\;\;\;\;\;\;\;\;\;\;\;=4\left[\left(\int_0^t\left(u,u_\tau\right)\operatorname d\tau\right)^2+\left(\int_0^t\left(\nabla u,\nabla u_\tau\right)\operatorname d\tau\right)^2+2\int_0^t\left(u,u_\tau\right)\operatorname d\tau\int_0^t\left(\nabla u,\nabla u_\tau\right)\operatorname d\tau\right]\\\;\;\;\;\;\;\;\;\;\;\;\;\;\;\leq 4\Big[\int_0^t\left\|u\right\|_2^2\operatorname d\tau\int_0^t\left\|u_\tau\right\|_2^2\operatorname d\tau+\int_0^t\left\|\nabla u\right\|_2^2\operatorname d\tau\int_0^t\left\|\nabla u_\tau\right\|_2^2\operatorname d\tau\\\;\;\;\;\;\;\;\;\;\;\;\;\;\;\;\;\;\;+2\left(\int_0^t\left\|u\right\|_2^2\operatorname d\tau\right)^\frac12\left(\int_0^t\left\|u_\tau\right\|_2^2\operatorname d\tau\right)^\frac12\left(\int_0^t\left\|\nabla u\right\|_2^2\operatorname d\tau\right)^\frac12\left(\int_0^t\left\|\nabla u_\tau\right\|_2^2\operatorname d\tau\right)^\frac12\Big]\\\;\;\;\;\;\;\;\;\;\;\;\;\;\;\leq 4\Big(\int_0^t\left\|u\right\|_2^2\operatorname d\tau\int_0^t\left\|u_\tau\right\|_2^2\operatorname d\tau+\int_0^t\left\|\nabla u\right\|_2^2\operatorname d\tau\int_0^t\left\|\nabla u_\tau\right\|_2^2\operatorname d\tau\\\;\;\;\;\;\;\;\;\;\;\;\;\;\;\;\;\;\;+\int_0^t\left\|u\right\|_2^2\operatorname d\tau\int_0^t\left\|\nabla u_\tau\right\|_2^2\operatorname d\tau+\int_0^t\left\|\nabla u\right\|_2^2\operatorname d\tau\int_0^t\left\|u_\tau\right\|_2^2\operatorname d\tau\Big)\\\;\;\;\;\;\;\;\;\;\;\;\;\;\;\leq4\left(\int_0^t\left\|u\right\|_2^2\operatorname d\tau+\int_0^t\left\|\nabla u\right\|_2^2\operatorname d\tau\right)\left(\int_0^t\left\|u_\tau\right\|_2^2\operatorname d\tau+\int_0^t\left\|\nabla u_\tau\right\|_2^2\operatorname d\tau\right)\\\;\;\;\;\;\;\;\;\;\;\;\;\;\;\leq4\int_0^t\left\|u\right\|_{H^1}^2\operatorname d\tau\int_0^t\left\|u_\tau\right\|_{H^1}^2\operatorname d\tau\leq4G\left(t\right)\int_0^t\left\|u_\tau\right\|_{H^1}^2\operatorname d\tau.\end{array}
\end{equation}
Considering (\ref{3.33}) and (\ref{3.34}),  it is easy to see that
\begin{equation*}\label{3.35}
\begin{array}{l}G''\left(t\right)G\left(t\right)-\frac{q+3}4\left(G'\left(t\right)\right)^2\geq G\left(t\right)\left(G''\left(t\right)-\left(q+3\right)\int_0^t\left\|u_\tau\right\|_{H^1}^2\operatorname d\tau\right)\\\;\;\;\;\;\;\;\;\;\;\;\;\;\;\;\;\;\;\;\;\;\;\;\;\;\;\;\;\;\;\;\;\;\;\;\;\;\;\;\;\;=G\left(t\right)\left(-2I\left(u\right)-\left(q+3\right)\int_0^t\left\|u_\tau\right\|_{H^1}^2\operatorname d\tau\right).\end{array}
\end{equation*}
Let
\begin{equation*}\label{3.36}
\xi\left(t\right)=-2I\left(u\right)-\left(q+3\right)\int_0^t\left\|u_\tau\right\|_{H^1}^2\operatorname d\tau.
\end{equation*}
Recalling (\ref{2.1}),  it follows that
\begin{equation*}\label{3.37}
J\left(u\right)=\left(\frac1p-\frac1{q+1}\right)\left\|\nabla u\right\|_p^p+\frac1{q+1}I\left(u\right),
\end{equation*}
Therefore,
\begin{equation*}\label{3.38}
\xi\left(t\right)=2\left(\frac{q+1}p-1\right)\left\|\nabla u\right\|_p^p-2\left(q+1\right)J\left(u_0\right)+\left(q-1\right)\int_0^t\left\|u_\tau\right\|_{H^1}^2\operatorname d\tau.
\end{equation*}
Now, we will discuss in two cases.

(i) If $0<J(u_0)<d$, from Lemma \ref{3.1} it follows that
\begin{equation}\label{3.39}
\xi\left(t\right)>2\left(q+1\right)d-2\left(q+1\right)J\left(u_0\right)+\left(q-1\right)\int_0^t\left\|u_\tau\right\|_{H^1}^2\operatorname d\tau=\rho>0.
\end{equation}
Then  we get
\begin{equation*}\label{3.40}
G''\left(t\right)G\left(t\right)-\frac{q+3}4\left(G'\left(t\right)\right)^2\geq\theta\rho>0,\;t\in\left[0,T_0\right],
\end{equation*}
which yields
\begin{equation*}\label{3.41}
\left(G^{-\vartheta}\left(t\right)\right)''=-\frac\vartheta{G^{\vartheta+2}\left(t\right)}\left(G''\left(t\right)G\left(t\right)-\left(\vartheta+1\right)\left(G'\left(t\right)\right)^2\right)<0,\;\vartheta=\frac{q-1}4.
\end{equation*}
Thus,  from the proof of \cite[Theorem 4.3]{30}, there exists a $T>0$ such that
\begin{equation*}\label{3.42}
\lim_{t\rightarrow T}G^{-\vartheta}\left(t\right)=0,
\end{equation*}
and
\begin{equation*}\label{3.43}
\lim_{t\rightarrow T}G\left(t\right)=+\infty,
\end{equation*}
which is a contradiction with $T=+\infty$.

(ii) If $J(u_0)\leq0$,  it is easy to get (\ref{3.39}) directly. The rest is proved similar to case (i).

The proof of Theorem \ref{theorem2} is complete.

\section{Critical initial energy $J\left(u_0\right)=d$}

In this section, we shall prove the global existence, uniqueness and decay estimate, and blowup of solutions to problem (\ref{1.1}) for the critical initial energy $J\left(u_0\right)=d$.

\textbf{Proof of Theorem \ref{theorem3}.}

\textbf{Global existence and uniqueness}. From the condition $J\left(u_0\right)=d$,  it follows that ${\left\|u_0\right\|}_{W^{1,p}}\neq0$. Denote $\mu_s=1-\frac1s$ and $u_{s0}=\mu_su_0$ for $s=2,3,\cdots$. We consider the problem (\ref{1.1}) with the condition
\begin{equation}\label{4.1}
u\left(x,0\right)=u_{s0}\left(x\right),\;s=2,3,\cdots.
\end{equation}

Coming back to Lemma \ref{2.3} and the initial data $I(u_0)\geq0$,  we obtain $\lambda^\ast\geq1$. Further, we deduce $I(u_{s0})=I(\mu_su_0)>0$ and $J(u_{s0})=J(\mu_su_0)<J(u_0)=d$. Due to Theorem \ref{theorem1}, the problem (\ref{1.1}) with the initial condition (\ref{4.1}) exists a unique global solution $u_s(t)\in L^\infty(0,\infty;W_N^{1,p}(\Omega))$ with $u_{st}(t)\in L^2\;(0,\infty;W_N^{1,2}(\Omega))$ for each $s=2,3,\cdots$. By Proposition \ref{2.1}, we can know that $u_s\left(x,t\right)\in W$. As in the proof of Theorem \ref{theorem1}, we readily get
\begin{equation*}\label{4.2}
\int_0^t\left\|u_{s\tau}\right\|_{H^1}^2\operatorname d\tau<d,\;\;\;0\leq t<\infty,
\end{equation*}
and
\begin{equation*}\label{4.3}
\left\|\nabla u_s\right\|_p^p<\frac{p\left(q+1\right)}{q+1-p}d,\;\;\;0\leq t<\infty.
\end{equation*}
The rest is proved similar to Theorem \ref{theorem1}.

\textbf{Asymptotic behavior}. The existence of the global solution $u(t)$ of (\ref{1.1}) is proved by the above, then we claim that $u\in W$ for $t>0$.

Arguing by contradiction, let $t_0>0$ is the first time that $I(u(t_0))=0$. Through the definition of $d$, we can know that $J(u(t_0))\geq d$. But
\begin{equation}\label{4.4}
0<J(u(t_0))=d-\int_0^{t_0}\left\|u_\tau\right\|_{H^1}^2\operatorname d\tau=d_1\leq d
\end{equation}
for any $t_0>0$.  Then, we have
\begin{equation}\label{4.5}
J(u(t_0))=d.
\end{equation}

Considering (\ref{4.4}) and (\ref{4.5}), gives
\begin{equation*}\label{4.6}
\int_0^{t_0}\left\|u_\tau\right\|_{H^1}^2\operatorname d\tau=0,
\end{equation*}
that is, $u_t\equiv0$ for $0\leq t\leq t_0$, which contradicts $I(u_0)>0$ (from (\ref{3.33})). Then, we infer $u\in W$ for $0<t<\infty$. Let $t_1>0$ as the initial time, it generates $u(x,t)\in W$ for $t>t_1$. Then by (\ref{3.23})-(\ref{3.25}) we can know that there exists a constant $\varsigma>0$ as
\begin{equation*}\label{4.7}
-\varsigma=C\left(C_\ast^{q+1}\left(\frac{p\left(q+1\right)}{q+1-p}J\left(u\left(t_1\right)\right)\right)^\frac{q-p+1}p-1\right)
\end{equation*}
such that
\begin{equation*}\label{4.8}
\left\|u\right\|_{H^1}^2<\left\|u\left(t_1\right)\right\|_{H^1}^2e^{-2\varsigma\left(t-t_1\right)},\;t\in\left(t_1,+\infty\right).
\end{equation*}
Theorem \ref{theorem3} is proved.

\textbf{Proof of Theorem \ref{theorem4}.}
By the continuity of $I(u)$ and $J(u)$ with respect to $t$, there exists a sufficiently small $t_0>0$ such that $I(u(t_0))<0$ and $J(u(t_0))>0$ for $I(u_0)<0$ and $J(u_0)=d>0$. Considering (\ref{3.33}),  it is easy to see that $u_t\neq0$ for $0<t\leq t_0$.  Then, we
have
\begin{equation*}\label{4.9}
J(u(t_0))=d-\int_0^{t_0}\left\|u_\tau\right\|_{H^1}^2\operatorname d\tau=d_0<d.
\end{equation*}
Applying Proposition \ref{2.1} and taking $t=t_0$ as the initial time, we obtain $u(x,t)\in V$ for $t>t_0$.

The rest is similar to the proof of Theorem \ref{theorem2}.

\section{Supercritical initial energy $J\left(u_0\right)>d$}

In this section, we shall prove the global existence, uniqueness and asymptotic behavior of solutions to problem (\ref{1.1}) for the supercritical initial energy $J\left(u_0\right)>d$ by the properties of $\omega$-limits of solutions.

\textbf{Proof of Theorem \ref{theorem5}.}

\textbf{Global existence and uniqueness}. Consider the basis $u_m\left(x,t\right)$ as in Theorem \ref{theorem1}, we claim that $u_m\in{\mathcal N}_+$.

Assuming that the assertion is not tenable, then there exists a $t_0>0$ such that $u_m\in{\mathcal N}_+$ for $t\in\left(0,t_0\right)$ and $u_m\left(t_0\right)\in{\mathcal N}$. By (\ref{3.22}), we have $\int_0^t\left\|u_{m\tau}\right\|_{H^1}^2\operatorname d\tau\neq0$ for $\Omega\times\left(0,t_0\right)$. Further, it follows from (\ref{3.4}) that $u_m(t_0)\in J^{J\left(u_m\left(0\right)\right)}$. Hence, $u_m(t_0)\in\mathcal N^{J\left(u_m\left(0\right)\right)}$. Due to the definition of $\lambda_{J(u_0)}$, we have
\begin{equation}\label{5.1}
  \left\|u_m\left(t_0\right)\right\|_{H^1}^2\geq\lambda_{J(u_0)}.
\end{equation}

Combining (\ref{3.22}), by the fact $I\left(u_m\right)>0$, gives
\begin{equation*}
\left\|u_m\left(t_0\right)\right\|_{H^1}^2<\left\|u_m\left(x,0\right)\right\|_{H^1}^2\leq\lambda_{J(u_0)},
\end{equation*}
which contradicts with (\ref{5.1}). Then, for all $t_0>0$, $u_m\in{\mathcal N}_+$.

Further, we can get $u_m\left(t\right)\in J^{J\left(u_m\left(0\right)\right)}\cap{\mathcal N}_+$. On the other hand, (\ref{3.4}) yields
\begin{equation*}
J\left(u_m\left(0\right)\right)\geq J\left(u_m\left(t\right)\right)\geq\left(\frac1p-\frac1{q+1}\right)\left\|\nabla u_m\right\|_p^p+\frac1{q+1}I\left(u_m\right),
\end{equation*}
and
\begin{equation*}
  J\left(u_m\left(0\right)\right)\geq\int_0^t\left\|u_{m\tau}\right\|_{H^1}^2\operatorname d\tau,
\end{equation*}
which imply boundedness of $\left\|\nabla u_m\right\|_p^p$ and $\int_0^t\left\|u_{m\tau}\right\|_{H^1}^2\operatorname d\tau$.

Similar to the proof of Theorem \ref{theorem1}, then problem (\ref{1.1}) has a unique global weak solution $u$.

\textbf{Asymptotic behavior}. We denote by $\omega\left(u_0\right)=\underset{t\geq0}\cap\overline{\left\{u\left(s\right):s\geq t\right\}}$ the $\omega$-limit of $u$. By (\ref{3.4}) and (\ref{3.22}), we have
\begin{equation*}
  \left\|\omega\right\|_{H^1}^2<\left\|u_0\right\|_{H^1}^2\leq\lambda_{J\left(u_0\right)},\;J\left(\omega\right)\leq J\left(u_0\right),
\end{equation*}
for any $\omega\in\omega\left(u_0\right)$, which means that $\omega\not\in\mathcal N^{J\left(u_0\right)}$ and $\omega\in J^{J\left(u_0\right)}$, i.e., $\omega\not\in\mathcal N$. Then, $\omega\left(u_0\right)\cap\mathcal N=\varnothing$, which implies $\omega\left(u_0\right)=\left\{0\right\}$. Therefore, $u\left(t\right)\rightarrow0$ as $t\rightarrow+\infty$.

Theorem \ref{theorem5} is proved.

\section{Conclusions}

In this paper, we've studied the global existence and finite time blowup of solutions for a mixed pseudo-parabolic $p$-Laplacian type equation. Compared with the research in \cite{15}, we replace the $\triangle u$ term in the equation with the $\mathrm{div}\left(\left|\nabla u\right|^{p-2}\nabla u\right)$ term according to actual physical background. By combining the Galerkin method and the theory of potential wells, we first present the explicit expression for the depth $d$ of potential well, and then overcome the difficulties in a priori estimation of $p$-Laplacian term and nonlinear term $\bbint_\Omega\left|u\right|^{q-1}u\operatorname dx$ in the equation. In Section 5, we overcome the limitation of the potential well family method in the $J\left(u_0\right)>d$ case by further analyzing the properties of $\omega$-limits of solutions. However, there exist many problems to be further studied such as in fractional order equations and others.



\section*{Availability of data and material}
  Not applicable.

\section*{Competing interests}
  The authors declare that they have no competing interests.

\section*{Funding}
This research was supported by the National Natural Science Foundation of China (No. 12071491).

\section*{Author's contributions}
    The authors contributed equally to the writing of this paper. All authors read and approved the final manuscript.


\section*{References}


\bibliographystyle{bmc-mathphys}

\end{document}